\documentclass[notheoremnums]{dpreprint}

\newtheorem{theorem}{Theorem}[section]
\newtheorem*{theorem*}{Theorem}
\newtheorem{corollary}[theorem]{Corollary}
\newtheorem*{corollary*}{Corollary}
\newtheorem{proposition}[theorem]{Proposition}
\newtheorem{lemma}[theorem]{Lemma}
\newtheorem*{maintheorem*}{Theorem}
\theoremstyle{definition}
\newtheorem{definition}[theorem]{Definition}

\begin{document}

\dtitle[Harmonic Functions on Groups]{Harmonic Functions on Compactly Generated Groups}
\dauthor[Darren Creutz]{Darren Creutz}{creutz@usna.edu}{U.S.~Naval~Academy}{Supported in part by a Naval Academy Research Council grant}
\datewritten{11 January 2022}

\dabstract{%
A compactly generated group is noncompact if and only if it admits a nonconstant harmonic function (for some, equivalently for every, reasonable measure).

This generalizes the known fact that a finitely generated group is infinite if and only if it admits a nonconstant harmonic function (for some, equivalently every, reasonable measure).
}

\makepreprint

\section*{Introduction}

Let $\Gamma$ be a finitely generated group and $S$ be a finite generating set.  A function $f : \Gamma \to \mathbb{R}$ is harmonic when for all $\gamma_{0} \in \Gamma$ we have $\frac{1}{|S|}\sum_{s\in S} f(\gamma_{0}s) = f(\gamma_{0})$.  It is a well-known result, going back to the work of Mok \cite{mok} and Kleiner \cite{kleiner}, that every infinite finitely generated group admits a nonconstant harmonic function; the converse, that a group admitting such a function is infinite, is immediate.

More recently, Tointon \cite{tointon} considered functions which are harmonic with respect to weighted measures: if $\mu : \Gamma \to [0,1]$ is a probability measure ($\sum \mu(\gamma) = 1$) that is symmetric ($\mu(\gamma^{-1}) = \mu(\gamma)$) with support generating the group then a $\mu$-harmonic $f : \Gamma \to \mathbb{R}$ is a function such that $\sum f(\gamma_{0}\gamma)\mu(\gamma) = f(\gamma_{0})$.  Tointon showed that $\Gamma$ is infinite if and only if there is some $\mu$ for which there is a nonconstant $\mu$-harmonic function if and only if there is a nonconstant $\mu$-harmonic function for every reasonable $\mu$.

We generalize this to the compactly generated case.  The fundamental ideas used are the same as those pioneered by Mok and refined by Kleiner (and indeed their proofs can be adapted for $\mu$-harmonic functions in the finitely generated case); our goal here is to present the complete argument collected in one place and to show how it extends to the compactly generated case:

\begin{maintheorem*}
A locally compact second countable compactly generated group $G$ is noncompact if and only if some (equivalently, every) reasonable probability measure $\mu$ on $G$ admits a nonconstant $\mu$-harmonic function.  Moreover, the set of such functions intersects nontrivially with the set of Lipschitz functions or the set of bounded functions (possibly with both).
\end{maintheorem*}

The definitions of reasonableness for a measure and $\mu$-harmonic functions appear in Section \ref{S:funcs} below.

\subsection*{On the Approach}

We employ the strategy used by Margulis \cite{Margulis}, namely that $G$ is compact if and only if it is amenable and has property $(T)$, following the approach taken by Mok \cite{mok} and Kleiner \cite{kleiner}.

Furstenberg \cite{furst} initiated the study of bounded harmonic functions on locally compact groups and Kaimanovich and Vershik \cite{KV} established that a group is amenable if and only if there exists a measure admitting nonconstant bounded harmonic functions.  Several nice accounts of this theory appear in the literature, e.g.~in \cite{BS06}, and our aim is to provide the companion work for the property $(T)$ half of the proof.

Making use of the notion of harmonic cocycles, specifically that property $(T)$ is equivalent to the vanishing of all nonconstant harmonic cocycles, we establish that the nonexistence of Lipschitz bounded nonconstant harmonic functions implies the group has property $(T)$.

The first systematic study of harmonic cocycles appears in the author's dissertation \cite{Creutz2011} and we include a presentation of the material relevant in that work here.  More recently, Bekka \cite{bekka} employed harmonic cocycles and their relationship to property $(T)$ in the setting of von Neumann algebras and Ozawa \cite{ozawa} and Erschler and Ozawa \cite{oe} studied the connection between harmonic cocycles and Shalom's property $H_{FD}$, also largely in the operator-algebraic setting.  Here we include the general theory of harmonic cocycles in terms of the energy functional and the direct construction of a sequence of cocycles tending towards harmonic.

\subsection*{Additional Results}

We use harmonic cocycles to give new (short) proofs of some results of Shalom on products of groups.

Specifically, we give a new proof of the fact that every cocycle $\beta$ on $G_{1} \times G_{2}$ is almost cohomologous to a cocycle of the form $\beta_{1} + \beta_{2} + \beta_{0}$ where $\beta_{0}$ takes values in the invariant vectors and each $\beta_{j}$ is a cocycle on $G_{j}$ under the representation restricted to the $G_{3-j}$-invariant vectors.

We also give a short, more direct, proof that if $\Gamma$ is a lattice in a product of simple compactly generated groups and $N \normal \Gamma$ then $\Gamma / N$ has property $(T)$.

\subsection*{Acknowledgments}

The author would like to thank the referee for pointing out an issue in the original formulation of the main theorem and for many other helpful suggestions.

\section{Harmonic Functions on Groups}\label{S:funcs}

Throughout $G$ is a locally compact second countable compactly generated group.

\begin{definition}
Let $\mu$ be a probability measure on $G$.  A function $f : G \to \mathbb{R}$ is \emph{$\mu$-harmonic} when for all $g \in G$ we have
\[
\int_{G} f(gh)~d\mu(h) = f(g)
\]
\end{definition}

For discrete groups, one is only interested in measures whose support generates the whole group (as otherwise one is in fact talking about harmonic functions on the subgroup generated by the support).  In the locally compact case, the analogous requirement is:

\begin{definition}
A probability measure $\mu$ on $G$ is \emph{admissible} when the support of $\mu$ generates $G$ and when $\mu$ (more generally some convolution power of $\mu$ with itself: $\mu * \mu * \cdots * \mu$) is nonsingular with respect to the Haar measure.
\end{definition}

One also prefers to work with symmetric measures:
\begin{definition}
A probability measure $\mu$ on $G$ is \emph{symmetric} when $d\mu(g) = d\mu(g^{-1})$, more precisely: for all measurable functions $f : G \to \mathbb{R}$ it holds that $\int f(g)~d\mu(g) = \int f(g^{-1})~d\mu(g)$.
\end{definition}

In order to formulate the argument, we also consider only measures with some control over their second moment relative to word length:
\begin{definition}
Let $K$ be a compact generating set for $G$ and write $|g|$ for the word length of $g \in G$ relative to the set $K$.
A probability measure $\mu$ on $G$ has \emph{finite second moment} when
\[
\int_{G} |g|^{2}~d\mu(g) < \infty
\]
(Note that while $|g|$ depends on the choice of $K$, the finiteness of the integral does not depend on the choice of generating set).
\end{definition}

With the definitions in mind, we formulate the class of measures to be considered:
  \begin{definition}
 A probability measure $\mu \in P(G)$ is \emph{reasonable} when it is admissible, symmetric and has finite second moment.
 \end{definition}
 
 The purpose of the paper is to prove:
 
  \begin{theorem}\label{T:propT}
 Let $G$ be a locally compact second countable compactly generated group.  If there exists a reasonable measure $\mu$ on $G$ such that every $\mu$-harmonic Lipschitz function is constant then $G$ has property $(T)$.
\end{theorem}
 
  \begin{corollary}\label{T:compactgroups}
 Let $G$ be a locally compact second countable compactly generated group.  
 The following are equivalent:
 \begin{itemize}
 \item $G$ is compact;
 \item there exists a reasonable measure $\mu$ on $G$ such that every $\mu$-harmonic Lipschitz function and every $\mu$-harmonic bounded function is constant; and
 \item for every reasonable measure $\mu$ on $G$, every $\mu$-harmonic Lipschitz function and every $\mu$-harmonic bounded function is constant.
 \end{itemize}
 \end{corollary}

\section{Unitary Representations}

Let $\mathcal{H}$ be a Hilbert space.  A continuous homomorphism $\pi : G \to \mathcal{U}(\mathcal{H})$ sending each element of $G$ to a unitary operator on $\mathcal{H}$ is a \emph{unitary representation of $G$ on a Hilbert space}.  

We begin by recalling the notions of cohomology and reduced cohomology.  The reader is referred to \cite{bdv} for further details on the well-known material in this section.

\subsection{Cocyles}

\begin{definition}
A map $b : G \to \mathcal{H}$ satisfying
\[
b(gh) = \pi(g)b(h) + b(g)
\]
for all $g,h \in G$ is a \emph{cocyle}.
\end{definition}

\begin{definition}
For $v \in \mathcal{H}$, the cocycle
\[
b(g) = \pi(g)v - v
\]
is a \emph{coboundary}
\end{definition}

The space of all cocycles of a representation is
\[
Z^{1}(G,\pi) = \{ b : G \to \mathcal{H}~|~b \text{ is a cocycle} \}
\]
and the space of coboundaries is
\[
B^{1}(G,\pi) = \{ b : G \to \mathcal{H}~|~b \text{ is a coboundary} \}
\]

\subsection{Cohomology}

\begin{definition}
The \emph{(first) cohomology} of the representation is
\[
H^{1}(G,\pi) = Z^{1}(G,\pi) / B^{1}(G,\pi)
\]
\end{definition}

\begin{definition}
A cocycle $b$ is \emph{cohomologous} to a cocycle $\varphi$ when they are in the same cohomology class, specifically: $b$ is cohomologous to $\varphi$ when there exists $v \in \mathcal{H}$ such that $b(g) = \varphi(g) + \pi(g)v - v$ for all $g \in G$.
\end{definition}

If we endow the space $Z^{1}(G,\pi)$ with the topology of strong convergence in $\mathcal{H}$ uniformly over compact sets in $G$, the space of coboundaries need not be closed.

\begin{definition}
The \emph{reduced (first) cohomology} of the representation is
\[
\overline{H^{1}}(G,\pi) = Z^{1}(G,\pi) / \overline{B^{1}(G,\pi)}
\]
\end{definition}

\begin{definition}
A cocycle $b$ is \emph{almost cohomologous} to a cocycle $\varphi$ when they are in the same reduced cohomology class, specifically: $b$ is almost cohomologous to $\varphi$ when there exists $v_{n} \in \mathcal{H}$ such that $b(g) = \lim_{n} \varphi(g) + \pi(g)v_{n} - v_{n}$ for all $g \in G$ where the limit is in the strong topology.
\end{definition}

In particular, since the map which is uniformly $0$ is a cocycle, we will be interested in cocycles which are \emph{almost cohomologous to $0$}: those $b \in Z^{1}(G,\pi)$ such that there exists $v_{n} \in \mathcal{H}$ with $b(g) = \lim \pi(g)v_{n} - v_{n}$ in the strong topology.

\subsection{Property $(T)$}

\begin{definition}
Let $\pi : G \to \mathcal{U}(\mathcal{H})$ be a unitary representation on a Hilbert space.  Then $\pi$ has \emph{almost invariant vectors} when there exist $v_{n} \in \mathcal{H}$ with $\| v_{n} \| = 1$ such that
\[
\| \pi(g)v_{n} - v_{n} \| \to 0
\]
uniformly over compact subsets of $G$.
\end{definition}

\begin{definition}
A group $G$ has property $(T)$ when for every unitary representation $\pi : G \to \mathcal{U}(\mathcal{H})$ with almost invariant vectors has a nonzero invariant vector: there exists $v \in \mathcal{H}$ with $\pi(g)v = v$ for all $g \in G$.
\end{definition}

The role of property $(T)$ in our study stems from it being in a certain sense disjoint from the notion of amenable, specifically:
\begin{theorem}
If $G$ is amenable and has property $(T)$ then $G$ is compact.
\end{theorem}

Reduced cohomology will be at the center of our presentation due to the following:
\begin{theorem}[Shalom \cite{shalom}]
A compactly generated locally compact second countable group $G$ has property $(T)$ if and only if for every irreducible unitary representation $\pi$ on a Hilbert space, the reduced cohomology is trivial.
\end{theorem}

\section{Harmonic Cocycles, Energy and Property $(T)$}\label{S:harmonic}

Throughout, $G$ is a locally compact second countable compactly generated group and $\mu$ is a reasonable probability measure on $G$.

\subsection{Harmonic Cocycles}

We now systematically study the space of harmonic cocycles; our ultimate aim being to show that they are essentially connected to reduced cohomology and thus to property $(T)$.

\begin{definition}
A cocycle $\beta : G \to \mathcal{H}$ is \textbf{$\mu$-harmonic} when for all $g \in G$,
\[
\beta(g) = \int_{G} \beta(gg^{\prime}) d\mu(g^{\prime}).
\]
\end{definition}

The goal of this section is to prove:
\begin{theorem}\label{T:harmoniccocycles}
Let $G$ be a locally compact second countable compactly generated group and $\mu$ a reasonable probability measure on $G$.  In every reduced cohomology class $[\beta] \in \overline{H^{1}}(G,\pi)$ there exists a unique $\mu$-harmonic representative.
\end{theorem}

Before proving this, we derive a consequence:

\begin{theorem}\label{T:Tharmonic}
Let $G$ be a locally compact second countable compactly generated group.  The following are equivalent:
\begin{itemize}
\item $G$ has property $(T)$;
\item there exists a reasonable probability measure $\mu$ on $G$ so that the only $\mu$-harmonic cocycle is $0$; and
\item for every reasonable probability measure $\mu$ on $G$, the only $\mu$-harmonic cocycle is $0$
\end{itemize}
\end{theorem}

In particular, the lack of property $(T)$ is equivalent to the existence of some unitary representation and some reasonable probability measure such that there is a nonzero harmonic cocycle.

\begin{proof}
Assume $G$ has property $(T)$.  Let $\mu$ be a reasonable probability measure on $G$ and $\varphi$ a $\mu$-harmonic cocycle for some representation $\pi : G \to \mathcal{U}(\mathcal{H})$.  Since $G$ has $(T)$, the only reduced cohomology class in $\overline{H^{1}}(G,\pi)$ is the equivalence class of $0$ (Theorem 4.2 \cite{Sha06}).  As $\varphi$ is a cocycle, it is then almost cohomologous to $0$.  However, both $\varphi$ and $0$ are $\mu$-harmonic so Theorem \ref{T:harmoniccocycles} implies that $\varphi = 0$.  Hence the first condition implies the third, which obviously implies the second.

Conversely, assume that $G$ does not have property $(T)$.  Then there is some unitary representation $\pi$ with nontrivial reduced cohomology.  Let $\varphi$ be a cocycle that is not almost cohomologous to zero and let $\mu$ be a reasonable measure.  By Theorem \ref{T:harmoniccocycles}, $\varphi$ is almost cohomologous to a $\mu$-harmonic cocycle $\beta$.  But $\beta$ cannot be zero since $\varphi$ is not almost cohomologous to zero.  Hence the second condition implies the first.
\end{proof}

We now turn to the proof of Theorem \ref{T:harmoniccocycles}.

\subsection{Uniqueness of Harmonic Cocycles}

First observe that:
\begin{lemma}\label{L:hare}
A cocycle $\beta$ is harmonic if and only if $\int_{G} \beta(g) d\mu(g) = 0$.
\end{lemma}
\begin{proof}
By the cocycle identity:
\[
\int_{G} \beta(gg^{\prime}) d\mu(g^{\prime}) = \int_{G} \beta(g) + \pi(g)\beta(g^{\prime}) d\mu(g^{\prime})
= \beta(g) + \pi(g) \int_{G} \beta(g^{\prime}) d\mu(g^{\prime}) \qedhere
\]
\end{proof}

If $\beta$ is a cocycle then the cocycle identity readily implies that
\[
\| \beta(g) \| \leq | g |~\sup_{s \in S} \| \beta(s) \|
\]
where $| g |$ is the word length in $G$ relative to a compact generating set $S$.  Since $\sup_{s \in S} \| \beta(s) \|$ is finite (as $\beta$ is continuous and $S$ is compact), we find that if $\mu$ has finite second moment then
\[
\int_{G} \|\beta(g)\|^{2}~d\mu(g) \leq C \int_{G} | g |^{2}~d\mu(g) < \infty
\]
and therefore Theorem \ref{T:harmoniccocycles} follows immediately from:

\begin{theorem}\label{T:harmoniccocycles2}
In any reduced cohomology class $[\beta] \in \overline{H^{1}}(G,\pi)$ that contains an $L^{2}(\mu)$ representative there exists a unique $\mu$-harmonic representative.
\end{theorem}

The remainder of the section is devoted to the proof.

\subsection{The Energy Function}

The key quantity in what follows is the energy of a cocycle.  The ideas involved in the study of energy decrementing go back to Mok \cite{mok} and play a key role in Kleiner's work \cite{kleiner}.

\begin{definition}
Let $G$ be a group and $\mu$ a probability measure on $G$.  Let $\pi : G \to B(\mathcal{H})$ be a unitary representation of $G$ on a Hilbert space $\mathcal{H}$.  Let $\beta : G \to \mathcal{H}$ be a cocycle.

The \textbf{$\mu$-energy of $\beta$} is defined to be
\[
E_{\beta}^{\mu} := \int_{G} \| \beta(g) \|^{2} d\mu(g)
\]
and the \textbf{$\mu$-energy function for $\beta$} is $E_{\beta}^{\mu} : \mathcal{H} \to [0,\infty]$ given by
\[
E_{\beta}^{\mu}(v) := \int_{G} \| \pi(g)v - v + \beta(g) \|^{2} d\mu(g).
\]
\end{definition}

The energy of a cocycle is the value of its energy function at $0$, that is $E_{\beta}^{\mu} = E_{\beta}^{\mu}(0)$, and in fact for $v \in \mathcal{H}$ the cocycle $\beta_{v}(g) = \beta(g) + \pi(g)v - v$ (which is cohomologous to $\beta$) has energy
\[
E_{\beta_{v}}^{\mu} = E_{\beta}^{\mu}(v)
\]

In order that the energy be finite, it is sufficient that $\mu$ be compactly supported (since $\| \pi(g)v - v + \beta(g)\| < \infty$ for each $g$).  More generally, since
\[
\| \pi(g)v - v + \beta(g) \| \leq 2 \| v \| + \| \beta(g) \|
\]
it is sufficient that $\| \beta(g) \| \in L^{2}(G,\mu)$.  The hypotheses of Theorem \ref{T:harmoniccocycles2} then ensure that $E_{\beta}^{\mu}$ is finite (which we will use implicitly throughout what follows).

\subsection{Properties of the Energy Function}

Let $\check{\mu}$ denote the symmetric opposite of $\mu$: $d\check{\mu}(g) = d\mu(g^{-1})$.
\begin{lemma}\label{L:5}
\begin{align*}
E_{\beta}^{\mu}(v &+ w) - E_{\beta}^{\mu}(v) = \int_{G} \| \pi(g)w - w \|^{2} d\mu(g)
- 2 \int_{G} \langle \pi(g)v - v + \beta(g), w \rangle d(\mu+\check{\mu})(g)
\end{align*}
\end{lemma}
\begin{proof}
\begin{align*}
E_{\beta}^{\mu}(v + w) - E_{\beta}^{\mu}(v) &= \int_{G} \| \pi(g)v - v + \beta(g) + \pi(g)w - w  \|^{2} - \| \pi(g)v - v + \beta(g) \|^{2} d\mu(g) \\
&= \int_{G} \| \pi(g)w - w \|^{2} + 2 \langle \pi(g)v - v + \beta(g), \pi(g)w - w \rangle d\mu(g) \\
\end{align*}
and using the cocycle identity
\begin{align*}
\langle \pi(g) v - v + \beta(g), \pi(g) w \rangle &= \langle v - \pi(g^{-1})v + \pi(g^{-1})\beta(g), w \rangle \\
&= \langle v - \pi(g^{-1})v + \beta(g^{-1}g) - \beta(g^{-1}), w\rangle \\
&= \langle v - \pi(g^{-1})v - \beta(g^{-1}), w \rangle
\end{align*}
since $\beta(e) = \beta(e) + \pi(e)\beta(e) = 2\beta(e)$ so $\beta(e) = 0$.
Now since
\begin{align*}
\int_{G} \langle \pi(g)v - v + \beta(g), \pi(g)w \rangle d\mu(g) 
&= \int_{G} \langle v - \pi(g^{-1})v - \beta(g^{-1}), w \rangle d\mu(g) \\
&= - \int_{G} \langle \pi(g)v - v + \beta(g), w \rangle d\check{\mu}(g)
\end{align*}
we have that
\begin{align*}
E_{\beta}^{\mu}(v + w) - E_{\beta}^{\mu}(v) &=
\int_{G} \| \pi(g)w - w \|^{2} d\mu(g) \\
&- 2 \int_{G} \langle \pi(g)v - v + \beta(g), w \rangle d\mu(g) - 2 \int_{G} \langle \pi(g)v - v + \beta(g), w \rangle d\check{\mu}(g). \qedhere
\end{align*}
\end{proof}

A useful special case is the following:
\begin{lemma}\label{L:actualprop}
\[
E_{\beta}^{\mu}(v) - E_{\beta}^{\mu} = \int_{G} \| \pi(g)v - v \|^{2} d\mu(g) - 2 \int_{G} \langle \beta(g) , v \rangle d(\mu + \check{\mu})(g)
\]
\end{lemma}

\begin{lemma}\label{L:energycont}
The energy function is continuous: if $\beta_{n}$ are cocycles such that $\beta_{n} \to \beta$ uniformly over compact sets in $G$ (and strongly in $\mathcal{H}$) and such that the $\beta_{n}$ and $\beta$ are uniformly in $L^{2}(\mu \times \| \cdot \|)$ (all dominated by the same $L^2$ function) then for all $v \in \mathcal{H}$
\[
E_{\beta}^{\mu}(v) = \lim_{n} E_{\beta_{n}}^{\mu}(v)
\]
and in particular
\[
E_{\beta}^{\mu} = \lim_{n} E_{\beta_{n}}^{\mu}.
\]
If $\mu$ is compactly supported then the $L^{2}(\mu)$ requirement is satisfied automatically for all cocycles and the energy function is continuous everywhere.
\end{lemma}
\begin{proof}
The uniform $L^{2}$ requirement (since $v$ is fixed $\beta_{n}(g) + \pi(g)v - v$ are also uniformly in $L^{2}$) allows us to apply the Dominated Convergence Theorem (twice) to conclude that 
\begin{align*}
E_{\beta}^{\mu}(v) &= \int_{G} \| \beta(g) + \pi(g)v - v \|^{2} d\mu(g) = \int_{G} \| \lim_{n} \beta_{n}(g) + \pi(g)v - v \|^{2} d\mu(g) \\
&= \int_{G} \lim_{n} \| \beta_{n}(g) + \pi(g)v - v \|^{2} d\mu(g) \\
&= \lim_{n} \int_{G} \| \beta_{n}(g) + \pi(g)v - v \|^{2} d\mu(g) = \lim_{n} E_{\beta_{n}}^{\mu}(v)
\end{align*}
When $\mu$ is compactly supported the energy of any cocycle is finite (since $\| \beta(g) \|$ must attain its maximum on the support of $\mu$) and since $\beta_{n} \to \beta$ uniformly on compact sets in $G$ the convergence is uniform on the support of $\mu$ meaning the Uniform Convergence Theorem can replace the need for the Dominated Convergence Theorem.
\end{proof}

\subsection{Directional Derivatives}

Consider the \textbf{directional derivative of the energy function}:
\[
D_{w}E_{\beta}^{\mu} = \lim_{t \to 0} \frac{E_{\beta}^{\mu}(tw) - E_{\beta}^{\mu}}{t}
\]

\begin{lemma}\label{L:dirderiv}
\[
D_{w}E_{\beta}^{\mu} = -2 \int_{G} \langle \beta(g), w \rangle d(\mu + \check{\mu})(g)
\]
\end{lemma}
\begin{proof}
Applying Lemma \ref{L:actualprop} to the directional derivative:
\begin{align*}
D_{w} E_{\beta}^{\mu}
&= \lim_{t\to 0} \frac{1}{t} \Big{(} \int_{G} t^{2} \| \pi(g)w - w \|^{2} d\mu(g)  - 2 \int_{G} t \langle \beta(g), w \rangle d(\mu+\check{\mu})(g) \Big{)} \\
& = -2 \int_{G} \langle \beta(g), w \rangle d(\mu + \check{\mu})(g). \qedhere
\end{align*}
\end{proof}

\begin{lemma}\label{L:diffderiv}
\[
E_{\beta}^{\mu}(v) - E_{\beta}^{\mu} = \int_{G} \| \pi(g)v - v \|^{2} d\mu(g) + D_{v}E_{\beta}^{\mu}.
\]
\end{lemma}
\begin{proof}
Lemma \ref{L:dirderiv} and Lemma \ref{L:actualprop}.
\end{proof}

\begin{lemma}\label{L:Dcw}
For $c \in \mathbb{R}$,
$D_{cv}E_{\beta}^{\mu} = c D_{v}E_{\beta}^{\mu}$.
\end{lemma}
\begin{proof}
By Lemma \ref{L:dirderiv} (twice),
\begin{align*}
D_{cv}E_{\beta}^{\mu} &= -2 \int_{G} \langle \beta(g), cv \rangle d(\mu + \check{\mu})(g) \\
&= -2 c \int_{G} \langle \beta(g), w \rangle d(\mu + \check{\mu})(g) = c D_{v}E_{\beta}^{\mu}. \qedhere
\end{align*}
\end{proof}

\begin{lemma}\label{L:derivenergycont}
The derivative of the energy function is continuous for each fixed direction: if $v \in \mathcal{H}$ and $\beta_{n}$ are cocycles such that $\beta_{n} \to \beta$ uniformly over compact sets in $G$ (and strongly in $\mathcal{H}$) and such that the $\beta_{n}$ and $\beta$ are uniformly in $L^{2}(\mu)$ (all dominated by the same $L^2$ function) then
\[
D_{v}E_{\beta}^{\mu} = \lim_{n} D_{v}E_{\beta_{n}}^{\mu}.
\]
If $\mu$ is compactly supported then the $L^{2}(\mu)$ requirement is satisfied automatically for all cocycles and the derivative of the energy function is continuous everywhere.
\end{lemma}
\begin{proof}
Fix $v \in \mathcal{H}$.
By Lemma \ref{L:diffderiv} we have that for any cocycle $\varphi$
\[
D_{v} E_{\varphi}^{\mu} = E_{\varphi}^{\mu}(v) - E_{\varphi}^{\mu} - \int_{G} \| \pi(g)v - v\|^{2} d\mu(g)
\]
and so
\[
\lim_{n} D_{v} E_{\beta_{n}}^{\mu} = \lim_{n} \big{(}E_{\beta_{n}}^{\mu}(v) - E_{\beta_{n}}^{\mu}\big{)} - \int_{G} \| \pi(g)v - v\|^{2} d\mu(g)
\]
and therefore, by continuity of energy (Lemma \ref{L:energycont}),
\[
\lim_{n} D_{v} E_{\beta_{n}}^{\mu} = E_{\beta}^{\mu}(v) - E_{\beta}^{\mu} - \int_{G} \| \pi(g)v - v\|^{2} d\mu(g)
= D_{v} E_{\beta}^{\mu}. \qedhere
\]
\end{proof}

\subsection{Minima of the Energy}

\begin{definition}
A cocycle $\beta$ \textbf{minimizes the energy} when
\[
E_{\beta}^{\mu} \leq E_{\beta}^{\mu}(v)
\]
for all $v \in \mathcal{H}$.
\end{definition}

\begin{lemma}\label{L:dirzero}
A cocycle $\beta$ minimizes the energy if and only if all the directional derivatives of the energy function for $\beta$ are zero.
\end{lemma}
\begin{proof}
Assume that $D_{v}E_{\beta}^{\mu} = 0$ for all $v \in \mathcal{H}$.  Then
by Lemma \ref{L:diffderiv},
\[
E_{\beta}^{\mu}(v) - E_{\beta}^{\mu} = \int_{G} \| \pi(g) v - v \|^{2} d\mu(g) + D_{v}E_{\beta}^{\mu} = \int_{G} \| \pi(g)v - v \|^{2} d\mu(g) \geq 0
\]
so $\beta$ minimizes the energy.

Now suppose that for some $v \in \mathcal{H}$ we have
$D_{v}E_{\beta}^{\mu} \ne 0$.  Replacing $v$ by $-v$ (Lemma \ref{L:Dcw}) if necessary, we may assume that there is $\delta > 0$ such that
\[
D_{v}E_{\beta}^{\mu} \leq - \delta
\]
From the definition of the directional derivative this means there is some $t > 0$ such that
\[
E_{\beta}^{\mu}(tv) - E_{\beta}^{\mu} \leq - t \delta
\]
meaning that $\beta$ does not minimize the energy.
\end{proof}

\subsection{Centers of Cocycles}

\begin{definition}
Let $\beta : G \to \mathcal{H}$ be a cocycle and $\mu$ a probability measure on $G$.  The \textbf{$\mu$-center of $\beta$} is the vector
\[
\mu(\beta) := \int_{G} \beta(g) d\mu(g).
\]
\end{definition}
 A cocycle is then $\mu$-harmonic if and only if its $\mu$-center is $0$ (Lemma \ref{L:hare}).

\begin{definition}
The \textbf{$\mu$-convolution of a vector $v$} is
\[
\pi(\mu)v := \int_{G} \pi(g)v d\mu(g).
\]
\end{definition}

\begin{lemma}\label{L:harmonicdecrease}
Let $\beta : G \to \mathcal{H}$ be a cocycle and let $\mu \in P(G)$.  Write $\overline{\mu} \in P(G)$ for the symmetrization of $\mu$: $\overline{\mu} = \frac{\mu + \check{\mu}}{2}$.  Then
\[
E_{\beta}^{\mu}(\overline{\mu}(\beta)) \leq E_{\beta}^{\mu}
\]
with equality if and only $\overline{\mu}(\beta)$ is a fixed point for the linear $(\pi)$ action of $G$ on $\mathcal{H}$.  If the only fixed point is zero then equality holds if and only if $\beta$ is $\overline{\mu}$-harmonic.
\end{lemma}
\begin{proof}
Note first that $\pi(g^{-1})\beta(g) = - \beta(g^{-1})$.  For $v \in \mathcal{H}$ and $g \in G$ (treating $\mathcal{H}$ as real; if it is complex the obvious modifications are required),
\begin{align*}
\| \beta(g) + &\pi(g)v - v \|^{2} - \| \beta(g) \|^{2} = \| \pi(g)v - v \|^{2} + 2 \langle \beta(g), \pi(g)v - v \rangle \\
&= \| \pi(g)v \|^{2} + \| v \|^{2} - 2 \langle \pi(g)v, v \rangle + 2 \langle \pi(g^{-1})\beta(g), v \rangle - 2 \langle \beta(g), v \rangle \\
&= 2 \| v \|^{2} - 2 \langle \pi(g)v, v \rangle - 2 \langle \beta(g^{-1}) + \beta(g), v \rangle
\end{align*}
and therefore
\[
E_{\beta}^{\mu}(v) - E_{\beta}^{\mu}
= 2 \| v \|^{2} - 2 \langle \pi(\mu)v, v \rangle - 2 \langle \mu(\beta) + \check{\mu}(\beta), v \rangle.
\]
Now observe that
\[
\| \pi(\mu)v \| \leq \int \| \pi(g)v \|~d\mu(g) = \| v \|
\]
and therefore
\[
E_{\beta}^{\mu}(v) - E_{\beta}^{\mu} \leq 4 \| v \|^{2} - 4 \big{\langle} \frac{\mu(\beta) + \check{\mu}(\beta)}{2}, v \big{\rangle}.
\]
Then taking $v = \frac{1}{2}(\mu(\beta) + \check{\mu}(\beta)) = \overline{\mu}(\beta)$ we obtain that
\[
E_{\beta}^{\mu}(v) - E_{\beta}^{\mu} \leq 0.
\]
For the inequalities above to all be equalities requires that $\pi(g) \overline{\mu}(\beta) = \pi(h) \overline{\mu}(\beta)$ for all $g,h$ in the support of $\mu$.  Hence in this case $\overline{\mu}(\beta)$ is a fixed point for the linear action.
\end{proof}

\subsection{Harmonic Cocycles Minimize Energy}

\begin{proposition}\label{P:minharmonic}
Let $\pi$ be a unitary representation of a group $G$ on a Hilbert space $\mathcal{H}$ with no nonzero fixed points and $\varphi : G \to \mathcal{H}$ a cocycle and let $\mu \in P(G)$ be symmetric.
Then $\varphi$ is $\mu$-harmonic if and only if it minimizes the $\mu$-energy.
\end{proposition}
\begin{proof}
Since $\mu$ is symmetric, $\overline{\mu} = \check{\mu} = \mu$.
Assume $\varphi$ minimizes the energy.  Then $E_{\varphi}^{\mu} \leq E_{\varphi}^{\mu}(v)$ for all $v \in \mathcal{H}$.  By Lemma \ref{L:harmonicdecrease} we have that $E_{\varphi}^{\mu}(\mu(\beta)) \leq E_{\varphi}^{\mu}$ and equality occurs if and only if $\varphi$ is harmonic.  Hence minimizing the energy implies $\varphi$ is harmonic.

Now assume that $\varphi$ is harmonic.  Since $\mu = \check{\mu}$ ($\mu$ is symmetric), by Lemma \ref{L:dirzero}, harmonicity implies $D_{v}E_{\varphi}^{\mu} = -4 \langle \mu(\beta), v \rangle = 0$ for all $v \in \mathcal{H}$.  Now by Lemma \ref{L:dirderiv}, for any $v$,
\[
E_{\varphi}(v) - E_{\varphi} = \int_{G} \| \pi(g)v - v \|^{2} + D_{v}E_{\varphi} = \int_{G} \| \pi(g)v - v \|^{2} d\mu(g) \geq 0
\]
and therefore $\varphi$ minimizes the energy.
\end{proof}

\subsection{Proof of Theorem \ref{T:harmoniccocycles2} When There Are No Fixed Points}

\begin{proposition}\label{P:uniquemin}
Let $\pi$ be a unitary representation of a group $G$ on a Hilbert space $\mathcal{H}$ and $\mu \in P(G)$ an admissible symmetric probability measure on $G$ (recall that admissible means the support generates the group and some convolution power is nonsingular with respect to Haar measure).

Let $[\beta]$ be a reduced cohomology class containing an $L^{2}(\mu)$ representative.  There exists a unique cocycle $\varphi \in [\beta]$ such that $\varphi$ minimizes the $\mu$-energy.
\end{proposition}
\begin{proof}
First we establish uniqueness.  Suppose $\varphi$ and $\phi$ are both cocycles in the same reduced cohomology class so that both minimize the energy.  Since they are in the same reduced cohomology class there exists a sequence $v_{n}$ such that $\phi(g) = \lim_{n} \varphi(g) + \pi(g)v_{n} - v_{n}$ (strongly) uniformly over compact sets in $G$.

Since $\varphi$ minimizes the energy we have $D_{v}E_{\varphi}^{\mu} = 0$ for all $v$ by Lemma \ref{L:dirzero}.  By Lemma \ref{L:actualprop}, for any $v$,
\[
E_{\varphi}^{\mu}(v) - E_{\varphi}^{\mu} = \int_{G} \| \pi(g)v - v \|^{2}~d\mu(g) + D_{v}E_{\varphi}^{\mu} = \int_{G} \| \pi(g)v - v \|^{2}~d\mu(g)
\]
By the continuity of energy (Lemma \ref{L:energycont}) we have that
\[
0 = E_{\phi}^{\mu} - E_{\varphi}^{\mu} = \lim_{n} E_{\varphi}^{\mu}(v_{n}) - E_{\varphi}^{\mu} = \lim_{n} \int_{G} \| \pi(g) v_{n} - v_{n} \|^{2}~d\mu(g)
\]
and therefore, since $\phi - \varphi$ is a cocycle and $\phi(g) - \varphi(g) = \lim_{n} \pi(g)v_{n} - v_{n}$ (uniformly over compact sets in $G$, strongly in $\mathcal{H}$), again by the continuity of energy (Lemma \ref{L:energycont})
\[
E_{\phi - \varphi}^{\mu} = \lim_{n} E_{0}^{\mu}(v_{n}) = 0
\]
and so
\[
\int_{G} \| \phi(g) - \varphi(g) \|^{2}~d\mu(g) = 0
\]
meaning that $\phi(g) = \varphi(g)$ for $\mu$-almost every $g$ which by continuity (and that the support of $\mu$ generates $G$) gives $\phi = \varphi$.

We now show that a minimizer in fact exists.  Fix a cocycle $\varphi$ and take a sequence of vectors $v_{n} \in \mathcal{H}$ such that
\[
E_{\varphi}^{\mu}(v_{n}) \downarrow \inf_{v \in \mathcal{H}} E_{\varphi}^{\mu}(v)
\]
Define
\[
\beta_{n}(g) = \varphi(g) + \pi(g)v_{n} - v_{n}
\]
so that the $\beta_{n}$ are cocycles cohomologous to $\varphi$.  By the Parallelogram Law (that $\| a + b \|^{2} + \| a - b \|^{2} = 2 \| a \|^{2} + 2 \| b \|^{2}$),
\begin{align*}
\frac{1}{2} \int_{G} \| &\beta_{n}(g) - \beta_{m}(g) \|^{2}~d\mu(g)\\
&= \frac{1}{2} \int_{G} 2 \| \beta_{n}(g) \|^{2} + 2 \| \beta_{m}(g) \|^{2} - \| \beta_{n}(g) + \beta_{m}(g) \|^{2}~d\mu(g) \\
&= E_{\varphi}^{\mu}(v_{n}) + E_{\varphi}^{\mu}(v_{m}) - 2 \int_{G} \| \frac{1}{2}\big{(}\beta_{n}(g) + \beta_{m}(g)\big{)}\|^{2}~d\mu(g) \\
&= E_{\varphi}^{\mu}(v_{n}) + E_{\varphi}^{\mu}(v_{m}) - 2 E_{\varphi}^{\mu}(\frac{v_{n}+v_{m}}{2}) \\
&\leq E_{\varphi}^{\mu}(v_{n}) + E_{\varphi}^{\mu}(v_{m}) - 2 \inf_{v\in\mathcal{H}} E_{\varphi}^{\mu}(v)
\end{align*}
and therefore
\[
\lim_{n,m \to \infty} \int_{G} \| \beta_{n}(g) - \beta_{m}(g) \|^{2}~d\mu(g) = 0
\]
meaning that for almost every $g$ we have that (again by Lemma \ref{L:energycont} we have continuity)
\[
\| \beta_{n}(g) - \beta_{m}(g) \| \to 0 \quad\quad \text{as $n,m \to \infty$}.
\]
Since this a Cauchy sequence of vectors (for each $g$) it has a strong limit point, call it $\beta(g)$:
\[
\beta(g) = \lim_{n \to \infty} \beta_{n}(g) = \lim_{n\to\infty} \varphi(g) + \pi(g)v_{n} - v_{n}.
\]
So $\beta$ is evidently a cocycle almost cohomologous to $\varphi$.
We remark at this point that were the infimum of the energy $\inf_{v\in\mathcal{H}} E_{\varphi}^{\mu}(v)$ attained by some vector $v$ then $\beta$ will simply be $\varphi(g) + \pi(g)v - v$.

The proof will be complete if we show that $\beta$ minimizes the energy, so it is enough (by Lemma \ref{L:dirzero}) to show that $DE_{\beta}^{\mu} = 0$.  Suppose that $D_{w}E_{\beta}^{\mu} \ne 0$ for some $w$.  Then, by the continuity of the derivative of the energy function (Lemma \ref{L:derivenergycont}),
\[
\lim_{n} D_{w}E_{\beta_{n}}^{\mu} \ne 0
\]
so there exists $\delta > 0$ and a subsequence of $n$ so that
\[
D_{w}E_{\beta_{n}}^{\mu} \leq - \delta
\]
(we may replace $w$ by $-w$ to force the derivative to be negative using Lemma \ref{L:Dcw}).

By Lemma \ref{L:Dcw}, for any $c \geq 0$,
\[
D_{cw}E_{\beta_{n}}^{\mu} \leq -c \delta
\]
and then by Lemma \ref{L:actualprop}
\[
E_{\beta_{n}}^{\mu}(cw) - E_{\beta_{n}}^{\mu} = \int_{G} \| \pi(g)cw - cw \|^{2}~d\mu(g) + D_{cw}E_{\beta_{n}}^{\mu} \leq 4c^{2}\|w\|^{2} - c\delta.
\]
Now $E_{\beta_{n}}^{\mu} = E_{\varphi}^{\mu}(v_{n}) \downarrow \inf_{v \in \mathcal{H}} E_{\varphi}^{\mu}(v)$ so for any $\epsilon > 0$ and large $n$,
\[
E_{\beta_{n}}^{\mu}(cw) - E_{\beta_{n}}^{\mu} \geq \inf_{v \in \mathcal{H}} E_{\varphi}^{\mu}(v) - \big{(} \inf_{v \in \mathcal{H}} E_{\varphi}^{\mu}(v) + \epsilon \big{)} = - \epsilon
\]
and therefore for any $\epsilon > 0$, by taking $c = \frac{\delta}{8 \| w \|^{2}}$, we have
\[
- \epsilon \leq c(4c \|w\|^{2} - \delta) = c \Big{(} \frac{\delta}{2} - \delta \Big{)} = - \frac{c\delta}{2} = - \frac{\delta^{2}}{16 \| w \|^{2}}.
\]
Hence $\delta^{2} \leq 16 \| w \|^{2} \epsilon$.  Taking $\epsilon \to 0$ then contradicts that $\delta > 0$ and so we conclude that in fact  $D_{v}E_{\beta}^{\mu} = 0$.
\end{proof}

\subsection{Completing the Proof -- Fixed Points}

\begin{proof}[Proof of Theorem \ref{T:harmoniccocycles2}]
We handle the case of $\mathcal{H}$ having fixed points under $\pi(G)$.  Decompose $\mathcal{H} = \mathcal{H}_{fixed} \oplus \mathcal{H}_{unfixed}$ where $\mathcal{H}_{fixed}$ is the set of fixed points for $\pi$.  These are clearly $G$-invariant subspaces of $\mathcal{H}$.

Let $\beta : G \to \mathcal{H}$ be a cocycle.  Write $\beta(g) = \beta_{fixed}(g) + \beta_{unfixed}(g)$ where $\beta_{fixed}(g)$ is the projection of $\beta(g)$ to $\mathcal{H}_{fixed}$ (and likewise for $\beta_{unfixed}$).  By the cocycle identity,
\begin{align*}
\beta_{fixed}(gh) &+ \beta_{unfixed}(gh) = \pi(g)\beta_{fixed}(h) + \pi(g)\beta_{unfixed}(h) + \beta_{fixed}(g) + \beta_{unfixed}(g)
\end{align*}
and therefore, since $\pi(g)\beta_{fixed}(h) = \beta_{fixed}(h)$,
\[
\beta_{fixed}(gh) - \beta_{fixed}(g) - \beta_{fixed}(h) = - \beta_{unfixed}(gh) + \pi(g)\beta_{unfixed}(h) + \beta_{unfixed}(g)
\]
The vector on the left is in $\mathcal{H}_{fixed}$ while that on the right is in $\mathcal{H}_{unfixed}$.  Therefore both are $0$.  So $\beta_{fixed} : G \to \mathcal{H}_{fixed}$ and $\beta_{unfixed} : G \to \mathcal{H}_{unfixed}$ are both cocycles.

Now $\beta_{fixed}(g^{-1}) = - \beta_{fixed}(g)$ by the cocycle identity and so
\[
\int_{G} \beta_{fixed}(g) d\mu(g) = 0
\]
since $\mu$ is symmetric.  Since $\pi$ is the trivial representation on $\mathcal{H}_{fixed}$ and there are no nontrivial coboundaries for the trivial representation, $\beta_{fixed}$ is the sole representative of its (reduced) cohomology class and is also $\mu$-harmonic.

The conclusion of Theorem \ref{T:harmoniccocycles2} holds for $\beta_{unfixed}$ as there are no fixed points (Proposition \ref{P:uniquemin}).  So there is a unique $\mu$-harmonic representative for $[\beta_{unfixed}]$.  Call this representative $\varphi_{unfixed}$.  Define the cocycle on $\mathcal{H}$
\[
\varphi := \beta_{fixed} + \varphi_{unfixed}
\]
Since $\beta_{fixed}$ and $\varphi_{unfixed}$ are $\mu$-harmonic, so is $\varphi$.  Since $\varphi_{unfixed} \in [\beta_{unfixed}]$ we clearly have $\varphi \in [\beta]$.  So $[\beta]$ has a $\mu$-harmonic representative.

Suppose $\phi$ were also a $\mu$-harmonic representative of $[\beta]$.  Decompose $\phi = \phi_{fixed} + \phi_{unfixed}$.  Observe that
\[
\phi(g) - \beta(g) = \lim \pi(g)v_{n} - v_{n}
\]
for some $v_{n}$ which we may take to be in $\mathcal{H}_{unfixed}$ (if $v_{n} = v_{n,fixed} + v_{n,unfixed}$ then $\pi(g)v_{n} - v_{n} = \pi(g)v_{n,unfixed} - v_{n,unfixed}$).  Then
\[
\phi_{fixed}(g) - \beta_{fixed}(g) = \beta_{unfixed}(g) - \phi_{unfixed}(g) + \lim \pi(g)v_{n} - v_{n}
\]
and the term on the left is in $\mathcal{H}_{fixed}$ while that on the right is in $\mathcal{H}_{unfixed}$.  So they are both $0$.  Then, $\phi_{fixed} = \beta_{fixed}$ and $\phi_{unfixed} \in [\beta_{unfixed}]$.  So, by the uniqueness of $\varphi_{unfixed}$ in $[\beta_{unfixed}]$ we have that $\phi = \beta_{fixed} + \varphi_{unfixed} = \varphi$.  So $\varphi$ is unique.
\end{proof}

\section{Characterizing Property $(T)$ and Compactness}
 
We now prove the main results:
\begin{proof}[Proof of Theorem \ref{T:propT}]
Suppose $G$ does not have property $(T)$.  Then by (the contrapositive of) Theorem \ref{T:Tharmonic} there exists $\pi : G \to \mathcal{U}(\mathcal{H})$ a unitary representation of $G$ on a Hilbert space and a nontrivial $\mu$-harmonic cocycle $\beta : G \to \mathcal{H}$.  For $v \in \mathcal{H}$ define the function
\[
\varphi_{v} : G \to \mathbb{R} \quad\quad\quad\quad \varphi_{v}(g) = \langle \beta(g), v \rangle
\]
and observe that
$|\varphi_{v}(g)| \leq \| \beta(g) \| \| v \|$.

Since $\beta$ is a cocycle,
\[
\| \beta(gh) \| = \| \pi(g)\beta(h) + \beta(g) \| \leq \| \pi(g)\beta(h) \| + \| \beta(g) \| = \| \beta(h) \| + \| \beta(g) \|.
\]
Let $K$ be a compact generating set for $G$ and set $L = \max \{ \| \beta(k) \| : k \in K \}$ which exists as $\beta$ is continuous and $K$ is compact.  Then, writing $|g|$ for the word length of $g$ relative to $K$, we have $\| \beta(g) \| \leq L |g|$ so
\[
\| \beta(g) - \beta(h) \| = \| \beta(hh^{-1}g) - \beta(h) \| = \| \pi(h)\beta(h^{-1}g) \| = \| \beta(h^{-1}g) \| \leq L |h^{-1}g|
\]
meaning $\varphi_{v}$ is Lipschitz.  Using the harmonicity of $\beta$,
\[
\int_{G} \varphi_{v}(gh)~d\mu(h) = \Big{\langle} \int_{G} \beta(gh)~d\mu(h), v \Big{\rangle} = \big{\langle} \beta(g), v \big{\rangle} = \varphi_{v}(g)
\]
so we conclude that $\varphi_{v}$ is $\mu$-harmonic.

By hypothesis this means that $\varphi_{v}$ is constant for every $v \in \mathcal{H}$.  As $\beta(e) = 0$ since it is a cocycle,
\[
\varphi_{v}(g) = \varphi_{v}(e) = 0
\]
for all $g \in G$ and all $v \in \mathcal{H}$.  But this in turn implies $\beta = 0$ contradicting that it is nontrivial.  Hence $G$ must in fact have property $(T)$.
\end{proof}

\begin{proof} [Proof of Corollary \ref{T:compactgroups}]
Lipschitz functions are always continuous since if $g_{n} \to g$ and $f$ is Lipschitz then
\[
|f(g_n) - f(g_m)| \leq C \| g_{m}^{-1}g_{n} \| \to 0\quad\quad\quad\quad\text{as $g_{m}^{-1}g_{n} \to e$}
\]

Assume $G$ is compact and let $\mu$ be any admissible measure on $G$.  Let $f$ be any Lipschitz $\mu$-harmonic function.  As $f$ is continuous, by the maximum principle $f$ attains its maximum so let $g_{0}$ such that
\[
f(g) \leq f(g_{0}) \quad\quad\text{for all $g \in G$}
\]
This implies that, using the harmonicity of $f$,
\[
f(g_{0}) = \int_{G} f(g_{0}h)~d\mu(h) \leq \int_{G} f(g_{0})~d\mu(h) = f(g_{0})
\]
which, by the convexity of the integral, means $f(g_{0}h) = f(g_{0})$ for almost every $h$, hence $f$ is constant $\mu$-almost everywhere.  As $f$ is continuous and the support of $\mu$ generates $G$ (and $\mu$ is nonsingular with respect to Haar measure), $f$ is constant.  So the first condition implies the third.

The third obviously implies the second.  Now assume that there is some admissible $\mu$ with finite second moment such that every $\mu$-harmonic Lipschitz function and every $\mu$-harmonic bounded function is constant.  As every $\mu$-harmonic bounded function is constant, $G$ is amenable by Furstenberg \cite{furst} and Kaimanovich-Vershik \cite{KV}.  Theorem \ref{T:propT} gives that $G$ has property $(T)$.  As any group with both property $(T)$ and amenability is compact, this shows that the second condition implies the first and completes the proof.
\end{proof}

\section{Reduced Cohomology of Products of Groups}

Our aim here is to use the theory of harmonic cocycles to give a new proof of a result of Shalom:
\begin{theorem}[Shalom \cite{Sha06}]\label{T:sha}
Let $G = G_{1} \times \cdots \times G_{k}$ be a product of at least two locally compact second countable, compactly generated groups and let $\pi$ be a unitary representation of $G$.

Every $(G,\pi)$-cocycle is almost cohomologous to a $(G,\pi)$-cocycle $\beta$ such that
\[
\beta = \beta_{0} + \beta_{1} + \cdots + \beta_{k}
\]
where $\beta_{0}$ takes values in the fixed points for $G$ and each $\beta_{j}$ is a $(G_{j},\pi\big{|}_{G_{j}})$-cocycle (treated as $(\prod_{j=1}^{k} G_{j},\pi)$-cocycle taking values in the $\check{G}_{j}$-invariant vectors where $\check{G}_{j} = \prod_{i\ne j} G_{i}$).
\end{theorem}

At the level of reduced cohomology:
\[
\overline{H^{1}}(\prod_{j=1}^{k} G_{j}, \pi) = \overline{H^{1}}(\prod_{j=1}^{k} G_{j}, \pi\big{|}_{fixed}) \oplus \bigoplus_{j=1}^{k} \overline{H^{1}}(G_{j},\pi\big{|}_{G_{j}}) 
\]

The rest of this section will occupy the proof; as the general case follows immediately from the $G_{1} \times G_{2}$ case, we prove it for that case.

\subsection{Cocycles on Product Groups}

For the remainder of this section, let $G_{1}, G_{2}$ be locally compact, second countable, compactly generated groups and set $G = G_{1} \times G_{2}$.  Let $\mu = \mu_{1} \times \mu_{2}$ where $\mu_{i}$ be reasonable probability measures on $G_{i}$.

For a unitary representation $\pi : G_{1} \times G_{2} \to \mathcal{H}$, one may consider the representation restricted to $G_{1}$ (or $G_{2}$) in the obvious manner: $\pi\big{|}_{G_{1}} : G_{1} \to \mathcal{U}(\mathcal{H})$ by $\pi\big{|}_{G_{1}}(g_{1}) = \pi(g_{1},e)$.

For a cocycle $\beta \in Z^{1}(G_{1} \times G_{2},\pi)$, we may likewise consider the restriction of $\beta$ to $G_{1}$ (or $G_{2})$ by
\[
\beta\big{|}_{G_{1}}(g_{1}) = \beta(g_{1},e)
\]
The reader may verify easily that $\beta\big{|}_{G_{1}}$ is indeed a cocycle for $\pi\big{|}_{G_{1}}$.

\subsection{Harmonic Functions and Fixed Points}

\begin{proposition}
Let $\beta$ be a $\mu$-harmonic cocycle for $(G, \pi)$ where $\pi$ is a unitary representation of $G_{1} \times G_{2}$.  Set
\[
v_{i} = \int_{G_{i}} \beta\big{|}_{G_{i}}(g_{i})~d\mu_{i}(g_{i})
\]
for $i=1,2$.  Then $v_{1}$ and $v_{2}$ (which need not be distinct) are fixed points for $(G,\pi)$.
\end{proposition}
\begin{proof}
Observe that
\begin{align*}
0 &= \int_{G} \beta(g)~d\mu(g)
= \int_{G_{2}}\int_{G_{1}} \beta(g_{1}g_{2})~d\mu_{1}(g_{1})~d\mu_{2}(g_{2}) \\
&= \int_{G_{2}} \int_{G_{1}} \beta(g_{1}) + \pi(g_{1})\beta(g_{2})~d\mu_{1}(g_{1})~d\mu_{2}(g_{2})
= v_{1} + \int_{G_{1}} \pi(g_{1}) v_{2}~d\mu_{1}(g_{1}) = v_{1} + \pi(\mu_{1})v_{2}
\end{align*}
and so $\pi(\mu_{1})v_{2} = -v_{1}$ (recall the convolution is defined as $\pi(\mu)v = \int \pi(g)v~d\mu(g)$).

Likewise, since the $G_{1}$ and $G_{2}$ actions commute, $\pi(\mu_{2})v_{1} = - v_{2}$.  Therefore
\[
\pi(\mu_{1} \times \mu_{2})v_{1} = \pi(\mu_{1})\pi(\mu_{2})v_{1} = \pi(\mu_{1})(-v_{2}) = v_{1}
\]
and so $\pi(\mu)v_{i} = v_{i}$ for $i = 1,2$.
Then $\varphi_{i}(g) = \pi(g)v_{i} - v_{i}$ are both $\mu$-harmonic coboundaries.  Since they are both almost cohomologous to $0$ they are both the unique harmonic representative of $[0]$ hence $\varphi_{1} = \varphi_{2} = 0$.  So $v_{1}$ and $v_{2}$ are fixed points.
\end{proof}

\subsection{Harmonic Functions on Products}

\begin{proposition}
Let $\beta$ be a $\mu$-harmonic cocycle for $(G_{1} \times G_{2},\pi)$.
Assume there are no nontrivial fixed points for $(G,\pi)$.  Then $\beta\big{|}_{G_{1}}$ takes values in the space of $\pi\big{|}_{G_{2}}$-invariant vectors and likewise for $\beta\big{|}_{G_{2}}$.
\end{proposition}
\begin{proof}
As there are no nontrivial $G$-fixed points in $\mathcal{H}$, $v_{1} = v_{2} = 0$ so
\[
\int_{G_{1}} \beta\big{|}_{G_{1}}~d\mu_{1} = v_{1} = 0
\]
meaning $\beta\big{|}_{G_{1}}$ is $\mu_{1}$-harmonic and likewise for $\beta\big{|}_{G_{2}}$.  Now for any $g_{2} \in G_{2}$,
\begin{align*}
\int_{G_{1}} \beta(g_{1} g_{2})~d\mu_{1}(g_{1}) &= \int_{G_{1}} \pi(g_{1}) \beta(g_{2})~d\mu_{1}(g_{1}) + \int_{G_{1}} \beta(g_{1})~d\mu_{1}(g_{1}) = \pi(\mu_{1})\beta(g_{2})
\end{align*}
and on the other hand, since the $G_{1}$ and $G_{2}$ actions commute,
\begin{align*}
\int_{G_{1}} \beta(g_{1} g_{2})~d\mu_{1}(g_{1})
&= \int_{G_{1}} \pi(g_{2}) \beta(g_{1})~d\mu_{1}(g_{1}) + \beta(g_{2})
= \pi(g_{2})v_{1} + \beta(g_{2}) = \beta(g_{2})
\end{align*}
and therefore $\beta(g_{2})$ is $\mu_{1}$-stationary for each $g_{2}$.  So each $\beta(g_{2})$ is a $G_{1}$-fixed point (by stationarity and the uniqueness of harmonic representatives).

Likewise each $\beta(g_{1})$ is a $G_{2}$-fixed point and so
\[
\beta(g_{1}, g_{2}) = \pi(g_{1})\beta(g_{2}) + \beta(g_{1}) = \beta(g_{2}) + \beta(g_{1})
\]
\[
\text{i.e.}\quad\beta = \beta\big{|}_{G_{1}} + \beta\big{|}_{G_{2}}
\]

Moreover, $\beta\big{|}_{G_{1}}$ takes values in the subspace of $G_{2}$-invariant vectors and likewise for $\beta\big{|}_{G_{2}}$.
\end{proof}

\subsection{The Nontrivial Fixed Points Case}

\begin{proposition}
Let $\beta$ be a $\mu$-harmonic cocycle for $(G_{1} \times G_{2},\pi)$.  Let $\mathcal{H}_{fixed}$ be the closed invariant subspace of $G$-invariant vectors.  Write $\beta_{fixed}$ for the projection of $\beta$ to $\mathcal{H}_{fixed}$.  Let $\mathcal{H}_{1}$ be the closed invariant subspace of $G_{2}$-invariant vectors and $\mathcal{H}_{2}$ be the closed invariant subspace of $G_{1}$-invariant vectors.

Then
\[
\beta = \beta_{fixed} + \beta_{1} + \beta_{2}
\]
where $\beta_{i}$ takes values in $\mathcal{H}_{i}$.
\end{proposition}
\begin{proof}
Dropping the assumption that there were no $G$-fixed points, and writing $\beta_{fixed}$ for $\beta$ projected onto the fixed points in $\mathcal{H}$, $\beta = \beta_{unfixed} + \beta_{fixed}$ where $\beta_{unfixed}$ is the projection onto the complement of the fixed points.

Since $\beta_{fixed}$ maps into invariant vectors and there are no nontrivial coboundaries for $\pi\big{|}_{fixed}$
 (and since $\beta_{0}(g^{-1}) = - \pi(g^{-1})\beta_{0}(g) = - \beta_{0}(g)$), $\mu$ being symmetric implies $\beta_{fixed}$ is $\mu$-harmonic.
 
 Therefore, every $\mu$-harmonic cocycle $\beta$ for $(G_{1}\times G_{2},\pi)$ is of the form $\beta = \beta_{fixed} + \beta\big{|}_{G_{1}} + \beta\big{|}_{G_{2}}$.  The claim then follows from the previous proposition.
 \end{proof}
 
 \subsection{Reduced Cohomology of Products}
 
By the uniqueness of harmonic cocycles, Theorem \ref{T:harmoniccocycles}, at the level of reduced cohomology this means:
\begin{theorem}
Let $\pi : G_{1} \times G_{2} \to \mathcal{U}(\mathcal{H})$ be a unitary representation of a product of groups.  Then
\[
\overline{H^{1}}(G_{1} \times G_{2}, \pi) = \overline{H^{1}}(G_{1},\pi\big{|}_{G_{1}}) \oplus \overline{H^{1}}(G_{2},\pi\big{|}_{G_{2}}) \oplus \overline{H^{1}}(G_{1} \times G_{2}, \pi\big{|}_{\mathcal{H}_{fixed}})
\]
\end{theorem}

\section{Quotients of Lattices}

\subsection{Lattices}

\begin{definition}
Let $G$ be a locally compact second countable group and $\Gamma < G$ a discrete subgroup (meaning that $\Gamma$ is a discrete set in the $G$-topology).  A \emph{fundamental domain} for $\Gamma$ is a set $F \subset G$ such that $F\Gamma = G$ and $\gamma~F \cap \gamma^{\prime}~F = \emptyset$ for $\gamma \ne \gamma^{\prime} \in \Gamma$.
\end{definition}

\begin{definition}
A discrete subgroup $\Gamma < G$ is a \emph{lattice} when there exists a fundamental domain $F$ with finite Haar measure: $\mathrm{Haar}_{G}(F) < \infty$.
\end{definition}

For $\Gamma < G$ a lattice, if $g \in G$ and $f \in F$ then $gf \in G$ so there is a unique element $\gamma \in \Gamma$ such that $gf\gamma \in F$ (this follows from the disjointness of the translates of $F$).

\begin{definition}
Let $\Gamma < G$ be a lattice.  The associated cocycle to a fundamental domain $F$ is $\alpha : G \times F \to \Gamma$ given by $\alpha(g,f)$ such that $gf\alpha(g,f) \in F$.
\end{definition}

Observe that for $g,h \in G$ and $f \in F$, $f_{0} = hf\alpha(h,f) \in F$ and so
\[
ghf\alpha(h,f)\alpha(g,hf\alpha(h,f)) = g(hf\alpha(h,f))\alpha(g,hf\alpha(h,f)) = gf_{0}\alpha(g,f_{0}) \in F
\]
meaning that
\[
\alpha(gh,f) = \alpha(h,f)\alpha(g,hf\alpha(h,f))
\]
which is why $\alpha$ is referred to as a cocycle (note: this is related to but not the same as the cocycle definition for representations).

\begin{definition}
A lattice $\Gamma < G$ is \emph{irreducible} when for every $N \normal G$, $N \ne G$, the projection of $\Gamma$ to $G/N$ is dense.

In particular, if $G_{1}$ and $G_{2}$ are simple then $\Gamma < G_{1} \times G_{2}$ is irreducible when $\overline{\mathrm{proj}_{G_{i}}~\Gamma} = G_{i}$.
\end{definition}

\subsection{Induced Representations}

\begin{definition}
A lattice is integrable when $\int_{F} |\alpha(g,f)|^{2}~dm(f) < \infty$.  (The finiteness of the integral does not depend on the choice of fundamental domain).
\end{definition}

Let $\Gamma$ be an integrable lattice in $G$.  Let $\pi : \Gamma \to \mathcal{U}(\mathcal{H})$ be a unitary representation of $\Gamma$ on a Hilbert space.  Fix a fundamental domain $F$ and corresponding cocycle $\alpha : G \times F \to \Gamma$.  Let $\tilde{\mathcal{H}} = L^{2}(F,\mathcal{H})$ be the Hilbert space of square integrable functions on $F$ taking values in $\mathcal{H}$.  Define
\[
\tilde{\pi} : G \to \tilde{\mathcal{H}} \quad\quad\text{by}\quad\quad (\tilde{\pi}(g)q)(f) = \pi(\alpha(g^{-1},f))~q(g^{-1}f\alpha(g^{-1},f))
\]
the reader may verify this is a unitary representation of $G$ (which requires the integrability condition on $\Gamma$).  For a cocycle $\varphi : \Gamma \to \mathcal{H}$, define
\[
\tilde{\varphi} : G \to \tilde{\mathcal{H}} \quad\quad\text{by}\quad\quad \tilde{\varphi}(g)(f) = \varphi(\alpha(g^{-1},f))
\]
the reader may verify that this is a cocycle and that the representation and cocycle are continuous (see Shalom \cite{shalom}).

\begin{definition}
The representation $\tilde{\pi}$ on $G$ is the \emph{induced representation} and the cocycle $\tilde{\varphi}$ is the \emph{induced cocycle}.
\end{definition}

The mapping $\varphi \mapsto \tilde{\varphi}$ induces an isomorphism of the cohomologies $H^{1}(\Gamma,\pi) \to H^{1}(G,\tilde{\pi})$ and of the reduced cohomologies $\overline{H^{1}}(\Gamma,\pi) \to \overline{H^{1}}(G,\tilde{\pi})$ (Proposition 1.11 \cite{shalom}).

\subsection{Harmonic Cocycles on Quotients of Lattices in Products}

We close the paper with a new result, Theorem \ref{T:new}, giving a new proof of a result of Shalom:
\begin{theorem}\label{T:new}
Let $N \normal \Gamma$ be a nontrivial normal subgroup of an irreducible integrable lattice $\Gamma < G = G_{1} \times \cdots \times G_{k}$ in a product of at least $k \geq 2$ simple compactly generated groups.  For every unitary representation of $\Gamma / N$ without invariant vectors, the only harmonic cocycle on $\Gamma / N$ (for any reasonable measure) is $0$.
\end{theorem}
\begin{proof}
Let $\pi_{0} : \Gamma / N \to \mathcal{U}(\mathcal{H})$ be a unitary representation without invariant vectors.  Let $\mu_{0}$ be a reasonable measure on $\Gamma / N$ and $\varphi_{0} : \Gamma / N \to \mathcal{H}$ a $\mu_{0}$-harmonic cocycle.

Define $\pi : \Gamma \to \mathcal{U}(\mathcal{H})$ by $\pi(\gamma) = \pi_{0}(\gamma N)$ and $\varphi : \Gamma \to \mathcal{H}$ by $\varphi(\gamma) = \varphi_{0}(\gamma N)$.

Let $\gamma_{1},\gamma_{2},\ldots$ be a system of representatives for $\Gamma / N$ and let $\sigma \in P(N)$ be a reasonable measure on $N$.  Define $\mu_{1} \in P(\Gamma)$ by
\[
\mu_{1}(\gamma_{i}n) = \mu_{0}(\gamma_{i}N)\sigma(n)
\]
and define $\mu = (1/2)(\mu_{1} + \check{\mu_{1}})$ so $\mu$ is symmetric.  As the support of $\mu_{0}$ generates $\Gamma / N$ and that of $\sigma$ generates $N$, the support of $\mu$ generates $\Gamma$.  Writing $m_{t}$ for the $t^{th}$ moment,
\begin{align*}
m_{2}(\mu_{1}) &= \sum_{i} \sum_{n} |\gamma_{i}n|^{2}\mu_{0}(\gamma_{i}N)\sigma(n) \\
&\leq \Big{(}\sum_{i}|\gamma_{i}|^{2}\mu_{0}(\gamma_{i}N)\Big{)}\Big{(}\sum_{n}\sigma(n)\Big{)} \\ &\quad\quad\quad\quad + 2\Big{(}\sum_{i}|\gamma_{i}|\mu_{0}(\gamma_{i}N)\Big{)}\Big{(}\sum_{n}|n|\sigma(n)\Big{)} + \Big{(}\sum_{i}\mu_{0}(\gamma_{i}N)\Big{)}\Big{(}\sum_{n}|n|^{2}\sigma(n)\Big{)} \\
&= m_{2}(\mu_{0}) + 2m_{1}(\mu_{0})m_{1}(\sigma) + m_{2}(\sigma) < \infty
\end{align*}
so $\mu_{1}$, hence $\mu$, has finite second moment, meaning $\mu$ is a reasonable measure on $\Gamma$.

Now observe that
\begin{align*}
\sum_{\gamma} \varphi(\gamma)\mu_{1}(\gamma) &=
\sum_{i} \sum_{n} \varphi(\gamma_{i}n_{j})\mu_{1}(\gamma_{i}n) \\
&= \sum_{i} \sum_{n} \varphi_{0}(\gamma_{i}N) \mu_{0}(\gamma_{i} N)\sigma(n) \\
&= \sum_{i} \varphi_{0}(\gamma_{i}N) \mu_{0}(\gamma_{i} N) = 0
\end{align*}
and likewise that
\begin{align*}
\sum_{\gamma} \varphi(\gamma)\check{\mu_{1}}(\gamma) &=
\sum_{i} \sum_{n} \varphi((\gamma_{i}n)^{-1})\mu_{1}(\gamma_{i}n) \\
&= \sum_{i} \sum_{n} \varphi(\gamma_{i}^{-1}n^{\prime})\mu_{1}(\gamma_{i}n) \\
&= \sum_{i} \sum_{n} \varphi_{0}(\gamma_{i}^{-1}N) \mu_{0}(\gamma_{i} N)\sigma(n) \\
&= \sum_{i} \varphi_{0}(\gamma_{i}^{-1}N) \mu_{0}(\gamma_{i} N) = 0
\end{align*}
since $\mu_{0}$ is symmetric.  Therefore $\varphi$ is $\mu$-harmonic.

Let $\tilde{\pi}$ be the induced representation.
Let $\tilde{\varphi}$ be the induced cocycle.  By Theorem \ref{T:sha}, $\tilde{\varphi}$ is almost cohomologous to an $\alpha_{1}\times\cdots\times\alpha_{k}$-harmonic $(G,\tilde{\pi})$-coycle $\beta$.  As there are no invariant vectors for $\pi$ there are none for $\tilde{\pi}$ so $\beta = \beta_{1} + \cdots + \beta_{k}$ where $\beta_{j}$ is a $G_{j}$-cocycle taking values in the $\check{G}_{j}$-invariant vectors.  

Since $\beta_{1}$ is a $(G,\tilde{\pi})$-cocycle, there exists a $(\Gamma,\pi)$-cocycle $b_{1}$ such that $\tilde{b_{1}}$ is cohomologous to $\beta_{1}$ as the mapping $b \mapsto \tilde{b}$ on cocycles is an isomorphism of cohomology.  So there exists $v \in \tilde{\mathcal{H}}$ such that for all $g$
\[
\beta_{1}(g) = \tilde{b_{1}}(g) + \tilde{\pi}(g)v - v
\]
meaning that for all $g$, for almost every $f \in F$,
\[
\beta_{1}(g)(f) + v(f) = \tilde{b_{1}}(g)(f) + \tilde{\pi}(g)v(f) = b_{1}(\alpha(g^{-1},f)) + \pi(\alpha(g^{-1},f))v(g^{-1}f\alpha(g^{-1},f))
\]

Define the measurable function $q_{0} : G \to \mathcal{H}$ by
\[
q_{0}(g) = \pi(\alpha(g^{-1},e))v(g^{-1}\alpha(g^{-1},e)) + b_{1}(\alpha(g^{-1},e))
\]
(defined for almost every $g$ by considering $f = g^{-1}\alpha(g^{-1},e)$).

Note that for $\gamma \in \Gamma$ we have $\alpha(g^{-1}\gamma^{-1},e) = \gamma\alpha(g^{-1},e)$ so
\[
q_{0}(\gamma g) = \pi(\gamma)q_{0}(g) + b_{1}(\gamma)
\]
Since
\[
\int \| q_{0}(f^{-1}) \|^{2}~dm(f) = \int \| v(f) \|^{2}~dm(f) = \| v \| < \infty
\]
 it follows that $q \in L^{\infty}(G,\mathcal{H})$.

Since $\pi(N) = \mathrm{id}$ and $b_{1}(N) = 0$ as $\pi$ is trivial on $N$, for $n \in N$ and $g \in G$ we have
\[
q_{0}(ng) = \pi(n)q_{0}(g) + b_{1}(n) = q_{0}(g)
\]

Define $q : G \to \mathcal{H}$ by taking $\phi : G \to [0,1]$ a continuous compactly supported function with $\int \phi~dh = 1$ and setting
\[
q(g) = \int q_{0}(gh)~\phi(h)dh
\]
so that $q$ is continuous.

Observe that for $g,h \in G$,
\begin{align*}
q_{0}(gh) &= \pi(\alpha(g^{-1},e))\pi(\alpha(h^{-1},g^{-1}\alpha(g^{-1},e)))v(h^{-1}g^{-1}\alpha(h^{-1}g^{-1},e)) \\
&\quad\quad\quad + \pi(\alpha(g^{-1},e))b_{1}(\alpha(h^{-1},g^{-1}\alpha(g^{-1},e))) + b_{1}(\alpha(g^{-1},e)) \\
&= \pi(\alpha(g^{-1},e))\big{[}\tilde{\pi}(h)v + \tilde{b_{1}}(h)\big{]}(g\alpha(g^{-1},e)) + b_{1}(\alpha(g^{-1},e))
\end{align*}
so for each $h \in G$, for almost every $g$,
\[
q_{0}(gh) = \pi(\alpha(g^{-1},e))\big{[}v + \beta_{1}(h)\big{]}(g^{-1}\alpha(g^{-1},e)) + b_{1}(\alpha(g^{-1},e))
\]

Since $\beta_{1}(g_{2}h) =  \beta_{1}(h)$ for $g_{2} \in \check{G}_{1}$, this means that for each $g_{2} \in \check{G}_{1}$ and almost every $g,h \in G$,
\begin{align*}
q_{0}(gg_{2}h) &= \pi(\alpha(g^{-1},e))\big{[}v + \beta_{1}(g_{2}h)\big{]}(g^{-1}\alpha(g^{-1},e)) + b_{1}(\alpha(g^{-1},e)) \\
&= \pi(\alpha(g^{-1},e))\big{[}v + \beta_{1}(h)\big{]}(g^{-1}\alpha(g^{-1},e)) + b_{1}(\alpha(g^{-1},e)) = q_{0}(gh)
\end{align*}
and so for almost every $g$,
\[
q(gg_{2}) = \int q_{0}(gg_{2}h)~\phi(h)dh = \int q_{0}(gh)~\phi(h)dh = q(g)
\]
As $q$ is continuous, then $q(gg_{2}) = q(g)$ for all $g \in G$ and $g_{2} \in \check{G}_{1}$ meaning that $q(g) = q(\mathrm{proj}_{1}~g)$.

As $N$ projects densely to $G_{1}$, for any $g \in G$ there exists $n_{\ell} \in N$ such that $\mathrm{proj}_{1}~n_{\ell} \to \mathrm{proj}_{1}~g$.  Since $q(ng) = \int q_{0}(ngh)~\phi(h)dh = \int q_{0}(gh)~\phi(h)dh = q(g)$ for $n \in N$ then, using that $q$ is continuous,
\[
q(g) = q(\mathrm{proj}_{1}~g) = \lim q(\mathrm{proj}_{1}~n_{\ell}) = \lim q(n_{\ell}) = q(e)
\]
meaning that $q$ is constant.  As $\phi$ was arbitrary, then $q_{0}$ must be constant almost everywhere.

Since $q_{0}(gh) = \pi(\alpha(g^{-1},e))\big{(}v + \beta_{1}(h)\big{)}(g\alpha(g^{-1},e)) + b_{1}(\alpha(g^{-1},e))$ for almost all $g,h$, this in turn means that $\beta_{1}(h)$ is constant over almost every $h$.  As it is harmonic, then $\beta_{1} = 0$.

The same holds for $\beta_{2}, \ldots, \beta_{k}$ by the same argument, and therefore $\tilde{\varphi}$ is almost cohomologous to $0$.  Since the map $\varphi \mapsto \tilde{\varphi}$ is an isomorphism of reduced cohomology, this means $\varphi$ is almost cohomologous to $0$.  But $\varphi$ and $0$ are both $\mu$-harmonic hence $\varphi = 0$ so $\varphi_{0} = 0$.
\end{proof}

Combining Theorem \ref{T:new} and Theorem \ref{T:Tharmonic} gives a new proof of:
\begin{corollary}[Shalom \cite{shalom}]
Let $\Gamma < G_{1} \times G_{2}$ be an irreducible lattice in a product of simple locally compact compactly generated groups and let $N \normal \Gamma$.  Then $\Gamma / N$ has property $(T)$.
\end{corollary}

\dbibliography{ReferenceList}

\end{document}